\documentclass{article}
\usepackage[latin1]{inputenc}
\usepackage[T1]{fontenc}
\usepackage{amsmath,amssymb}
\usepackage{amsfonts,amsthm}
\usepackage{geometry}
\usepackage{color}

\usepackage[plainpages=false,colorlinks=true, breaklinks=true]{hyperref}
\definecolor{darkblue}{rgb}{0,0,.5}
\hypersetup{linkcolor=darkblue, citecolor=red, menucolor=darkblue, urlcolor=darkblue}

\newcommand{\Cb}{\mathbb{C}}
\newcommand{\Nb}{\mathbb{N}}
\newcommand{\Zb}{\mathbb{Z}}

\newcommand{\Xb}{\textbf{\upshape X}}
\newcommand{\Yb}{\textbf{\upshape Y}}

\newcommand{\Kc}{\mathcal{K}}
\newcommand{\Lc}{\mathcal{L}}
\newcommand{\Pc}{\mathcal{P}}

\DeclareMathOperator{\BDO}{BDO}

\DeclareMathOperator{\coker}{coker}
\DeclareMathOperator{\diag}{diag}
\DeclareMathOperator{\diam}{diam}
\DeclareMathOperator{\im}{im}
\DeclareMathOperator{\op}{op}
\DeclareMathOperator{\rk}{rank}
\DeclareMathOperator{\plimn}{\underset{n\to\infty}{\Pc-lim \,}}
\DeclareMathOperator{\supp}{supp}

\providecommand{\lb}[1]{\Lc(#1)}
\providecommand{\lc}[1]{\Kc(#1)}
\providecommand{\pb}[1]{\Lc(#1,\Pc)}
\providecommand{\pc}[1]{\Kc(#1,\Pc)}

\newtheorem{thm}{Theorem}[section]
\newtheorem{prop}[thm]{Proposition}
\newtheorem{cor}[thm]{Corollary}
\theoremstyle{definition}
\newtheorem{defn}[thm]{Definition}
\newtheorem{rem}[thm]{Remark}
\newtheorem{ex}[thm]{Example}
\numberwithin{equation}{section}

\begin{document}

\title{On Semi-Fredholm Band-Dominated Operators}
\date{}
\author{Markus Seidel}
\maketitle

\begin{abstract}
In this paper we study the semi-Fredholm property of band-dominated operators $A$ 
and prove that it already implies the Fredholmness of $A$ in all cases 
where this is not disqualified by obvious reasons. Moreover, this observation
is applied to show that the Fredholmness of a band-dominated 
operator already follows from the surjectivity of all its limit operators.

\medskip


\medskip


\end{abstract}

\section{Introduction}
We discuss the question whether or when the semi-Fredholm property of a 
band-dominated operator on $\Xb=l^p(\Zb^N,X)$ already yields its Fredholmness. 
For classes of Toeplitz operators this is well known (see e.g. \cite[2.29, 2.47]{Analysis} 
and \cite[4.47, 4.108]{NumAna}). In particular, the abstract theory of continuous 
functions of shifts in \cite{GoFe} and its applications e.g. \cite[Section 4]{NumAna} 
typically provide results of such type.

Here, as usual, $l^p(\Zb^N,X)$, $1\leq p<\infty$, denotes the space of all 
functions/generalized sequences  $x=(x_n):\Zb^N\to X$ with values in a Banach 
space $X$ such that $\|x\|_p^p:=\sum_{n\in\Zb^N}\|x_n\|_X^p<\infty$.
Further $l^\infty(\Zb^N,X)$ is the space of bounded sequences, i.e. 
$\|x\|_\infty:=\sup\{\|x_n\|_X:n\in\Zb^N\}<\infty$, and $l^0(\Zb^N,X)\subset l^\infty(\Zb^N,X)$ 
is the subspace of all sequences with $\|x_n\|_X\to 0$ as $|n|\to\infty$.

In order to define band-dominated operators two basic families of generators are 
required: 
\begin{itemize}
\item Every sequence $a=(a_n)\in l^\infty(\Zb^N,\lb{X})$ defines a so called
multiplication operator $aI$ on $\Xb$ by the rule $aI:(x_n)\mapsto (a_nx_n)$.
\item Further, for every $k\in\Zb^N$, define the so-called shift operator $V_k$ on $\Xb$
as $V_k(x_n):=(x_{n-k})$.  
\end{itemize}
Now, finite combinations $\sum_\alpha a_\alpha V_\alpha$ of 
these generators are usually referred to as band operators, and the elements in the
closure $\BDO(\Xb)$ (w.r.t. the operator norm) of the set of all band operators are called
band-dominated operators. This set $\BDO(\Xb)$ actually forms a closed subalgebra of $\lb{\Xb}$, 
the Banach algebra of all bounded linear operators on $\Xb$, and covers a lot of more particular 
operator classes such as Toeplitz and Laurent operators with continuous symbols, discrete 
Schrödinger and Jacobi operators, and by the particular choice $X=L^p[0,1]$ also 
classes of convolution type integral operators 
\cite{LangeR, RochFredTh, RaRoSiL2, LiSi, Standard, LimOps, Marko, LiChW, LiSe, MaSaSe}.

In the present paper we make the following observation:
If $X$ is an infinite-dimensional Banach space or if $N>1$ then the answer to our 
question is ``No'', as the almost trivial counterexamples \ref{E1}, \ref{E2} will show. 
Surprisingly, the answer is ``Yes'' on all $l^p(\Zb,X)$, $p\in \{0\}\cup[1,\infty]$,
$\dim X<\infty$. 

Furthermore, this result will help to simplify the well known Fredholm criteria
for band-dominated operators in terms of limit operators 
\cite{LangeR, LimOps, Marko, LiChW, LiSe}:
In the above mentioned cases, if all limit operators of a band-dominated operator 
$A$ are invertible from one (and the same) side then $A$ is Fredholm.

\medskip

The paper is organized as follows: In Sections \ref{SAppN} and \ref{SPTheory} the 
essential tools and techniques are introduced and shortly discussed: the lower 
approximation numbers of an operator, and appropriate notions of compactness and 
Fredholmness which permit to study our problems on $l^p$ spaces for \textit{all} $p$ 
in an elegant and homogenous way. Section \ref{SSemiF} is devoted to the main theorem
on automatic Fredholmness and Section \ref{SOpSpec} finally addresses 
the characterization via limit operators.

\section{Semi-Fredholm operators and approximation numbers}\label{SAppN}

Let $\Xb$, $\Yb$ be Banach spaces. A (bounded linear) operator $A\in\lb{\Xb,\Yb}$ 
is said to be 
normally solvable if its range $\im A$ is closed. It is well known that $A$ is 
normally solvable if and only if its adjoint $A^*$ is normally solvable. In that 
case for the dimensions of the respective kernels and cokernels it holds that
$\dim\ker A=\dim\coker A^*$ as well as $\dim\coker A=\dim\ker A^*$. These 
observations can be found in any textbook on the subject, e.g. \cite{Kato, Aiena}.

\subparagraph{(Semi-)Fredholm operators}
An operator $A$ is called $\Phi_+$-operator if it is normally solvable
and its kernel is finite-dimensional. 
An operator $A$ is called $\Phi_-$-operator if it is normally solvable
and its cokernel is finite-dimensional.  
Operators in $\Phi_+\cup\Phi_-$ are referred to as semi-Fredholm operators, 
and those in $\Phi:=\Phi_+\cap\Phi_-$ are Fredholm operators.

\medskip

With the above we have that $A$ is $\Phi$ ($\Phi_+$, $\Phi_-$) if and only if 
$A^*$ is $\Phi$ ($\Phi_-$, $\Phi_+$, respectively). Moreover, notice that 
operators $A$ with  $\dim\coker A<\infty$ are automatically normally solvable.

\medskip

Next, we want to continue discussing further tools which permit to study the
Fredholm and semi-Fredholm properties of operators $A \in \lb{\Xb, \Yb}$ and
offer another perspective.

\subparagraph{Lower approximation numbers}
For $A \in \lb{\Xb, \Yb}$ the $m$-th approximation number from the right $s^r_m(A)$ 
and the $m$-th approximation number from the left $s^l_m(A)$ are defined as 
\begin{align*}
	s_m^r(A)& := \inf \{\|A-F\|_{\lb{\Xb,\Yb}}: F\in\lb{\Xb,\Yb}, \dim \ker F \geq m\},\\
	s_m^l(A)& := \inf \{\|A-F\|_{\lb{\Xb,\Yb}}: F\in\lb{\Xb,\Yb}, \dim \coker F \geq m\},
\end{align*}
$(m=0,1, 2,...)$, respectively. It is clear that
$0=s^r_0(A)\leq s^r_1(A)\leq s^r_2(A)\leq ...$ and that the same holds true for
the sequence $(s_m^l(A))_m$.

\subparagraph{Lower Bernstein and Mityagin numbers}
Denote by $U_\Xb$ the closed unit ball in $\Xb$ and by 
\begin{align*}
j(A) & := \sup\{\tau \geq 0 : \|Ax\|\geq \tau\|x\| \text{ for all } x\in\Xb\},\\
q(A) & := \sup\{\tau \geq 0 : A(U_\Xb) \supset \tau U_\Yb\}
\end{align*}
the injection modulus and the surjection modulus of an operator $A\in\lb{\Xb,\Yb}$, 
respectively. Obviously,
the equivalent characterization $j(A)=\inf\{\|Ax\|:x\in\Xb, \|x\|=1\}$ holds, 
and therefore $j(A)$ is often called the minimum modulus or the lower norm of $A$.
From \cite{Pietsch}, B.3.8 we know
\begin{equation}\label{EDualMod}
j(A^*)=q(A) \quad\text{and}\quad q(A^*)=j(A).
\end{equation}
Furthermore, for given closed subspaces $V\subset\Xb$ and $W\subset\Yb$ let 
$A|_V:V\to\Yb$ denote the restriction of $A$ to $V$, and by $\pi_W$ denote the 
canonical map of $\Yb$ onto the quotient $\Yb/W$. Define the lower Bernstein 
and Mityagin numbers by
\begin{align*}
B_m(A) & := \sup\{j(A|_V): \dim \Xb/V < m\},\\
M_m(A) & := \sup\{q(\pi_WA): \dim W < m\}.
\end{align*}
These characteristics have been discussed in \cite{Zeman}, for instance.
Note that the sequences $(B_m(A))$, $(M_m(A))$ are monotonically non-decreasing, too. 
Furthermore by \cite[Equation (2.6)]{SeSi3}
\begin{equation}\label{EDualBM}
M_m(A)=B_m(A^*) \quad\text{for every}\quad m\in \Nb.
\end{equation}

\begin{prop}\label{PANBound} (cf. \cite[Proposition 2.9]{SeSi3})
Let $A\in\lb{\Xb,\Yb}$. Then, for all $m\in\Nb$,
\begin{align*}
\frac{s^r_m(A)}{2^m-1} \leq B_m(A) \leq s^r_m(A),\quad \text{as well as}\quad
\frac{s^l_m(A)}{2^m-1} \leq M_m(A) \leq s^l_m(A).
\end{align*}
\end{prop}
Also notice that in case of $\Xb$, $\Yb$ being Hilbert spaces one even has
$s^r_m(A) = B_m(A) = \sigma_m(A)$ and $s^l_m(A) = M_m(A) = \sigma_m(A^*)$, where
$\sigma_m(A)$ denotes the $m$th lower singular value of $A$, resp. (cf. 
\cite[Corollary 2.12]{SeSi3}).
These geometric characteristics and the latter estimates which connect them
now permit to look at and characterize the Fredholm properties of operators by
another perspective.
\begin{cor}\label{CANNormSolv}
Let $A\in\lb{\Xb,\Yb}$.
\begin{enumerate}
\item If $A$ is normally solvable and $m>\dim\ker A$ then $s_m^r(A)>0$.
\item If $A$ is normally solvable and $m>\dim\coker A$ then $s_m^l(A)>0$.
\item If $A$ is normally solvable and $m\leq\dim\ker A\leq\infty$ then $s_m^r(A)=0$.
\item If $A$ is normally solvable and $m\leq\dim\coker A\leq\infty$ then $s_m^l(A)=0$.
\item If $A$ is not normally solvable then $s_m^r(A)=s_m^l(A)=0$ for 
			every $m\in\Nb$.
\end{enumerate}
\end{cor}
\begin{proof}
Let $A$ be normally solvable, $m<\infty$, and $\Xb_1$ be a $m$-dimensional subspace 
of $\ker A$. Then every subspace $V$ of $\Xb$ with $\dim \Xb/V < m$ has a non-trivial 
intersection with $\Xb_1$, hence $B_m(A)=0$.
If $\dim\ker A<m<\infty$ then there exists a complement $\Xb_2$ of $\ker A$ and
$A:\Xb_2\to\im A$ is a bijection between Banach spaces, hence $j(A|_{\Xb_2})>0$. 
Thus, $B_m(A)>0$. The previous proposition then yields the claims 
$\textit{1.}$ and $\textit{3.}$

If $A$ is not normally solvable, then for every $m$ and every $\delta>0$ there is a 
subspace $\Xb_1$ of $\Xb$, $\dim\Xb_1=m$, $\|A|_{\Xb_1}\|<\delta$ (see e.g. the proof 
of \cite[Theorem 3]{SeSi2}). Every subspace $V$ with $\dim \Xb/V < m$ has a non-trivial 
intersection with $\Xb_1$, hence $B_m(A)\leq\delta$. With the previous proposition 
and since $\delta$ is arbitrary, we see that $s_m^r(A)=0$.

For the analogous claims on the approximation numbers from the left we just recall 
that $A$ is normally solvable if and only if $A^*$ is normally solvable, that 
$M_m(A)=B_m(A^*)$, and that in case of normally solvable $A$ it holds that 
$\dim\coker A=\dim\ker A^*$.
\end{proof}

\begin{cor}
Let $A\in\lb{\Xb,\Yb}$.
\begin{enumerate}
\item If $s_m^r(A)>0$ for some $m$ then $A\in\Phi_+$.
\item If $s_m^l(A)>0$ for some $m$ then $A\in\Phi_-$.
\end{enumerate}
\end{cor}

\section{Some auxiliary results in the \texorpdfstring{$\Pc$}{P}-framework}\label{SPTheory}
During the last decades it turned out to be extremely fruitful to embed the class
of band-dominated operators in another Banach algebraic framework which provides
appropriate adapted notions of compactness, Fredholmness and convergence. We refer to 
\cite{LangeR, NonStrongly, RochFredTh, LimOps, Marko, LiChW, SeSi2, SeSi3} and the 
recent survey article \cite{SeFre} for the details and further references. Here and in 
Section \ref{SOpSpec} we just cite the basic definitions and borrow some of the 
tools that we are going exploit later on:

\begin{defn}
Introduce the collection $\Pc=(P_n)$ of the canonical projections 
$P_n:=\chi_{\{-n,\ldots,n\}^N}I$ on $\Xb=l^p(\Zb^N,X)$,\footnote{Here $\chi_U$ denotes
the characteristic function of the set $U\subset\Zb^N$.} and denote the respective
complementary projections by $Q_n:=I-P_n$.

A bounded linear operator $K$ is said to be $\Pc$-compact if
$\|KQ_n\|+\|Q_nK\|$ tends to zero as $n$ goes to infinity (that is, the action of 
$K$ is mainly concentrated near the origin of $\Zb^N$). 
The set of all $\Pc$-compact operators is denoted by $\pc{\Xb}$.

By $\pb{\Xb}$ denote the family 
$\pb{\Xb}:=\{A\in\lb{\Xb}: K\in\pc{\Xb}\Rightarrow AK,KA\in\pc{\Xb}\}$.
\end{defn}

Roughly speaking, one may say that $\pb{\Xb}$ collects the operators which are 
compatible with the new family of $\Pc$-compact ones since it indeed turns out that 
the set $\pb{\Xb}$ is a unital Banach algebra and contains $\pc{\Xb}$ as a closed 
two-sided ideal. Moreover, a bounded linear operator $A$ belongs to $\pb{\Xb}$ if 
and only if (see \cite[Section 1.1]{LimOps})
\begin{equation}\label{ELXP}
\|P_kAQ_n\| + \|Q_nAP_k\|\to 0\quad\text{as}\quad n\to\infty \quad\text{for every $k\in\Nb$}.
\end{equation}
This, together with the fact that 
$\|P_nKP_n-K\|\leq\|P_nK-K\|\|P_n\|+\|KP_n-K\|\to_n 0$ for every $K\in\pc{\Xb}$, 
easily show that in particular $\pc{\Xb}\subset\BDO(\Xb)\subset\pb{\Xb}$ holds.

\begin{defn}
An operator $A\in\pb{\Xb}$ is said to be $\Pc$-Fredholm if there exists a bounded
linear operator $B$ (a so-called $\Pc$-regularizer for $A$) such that $I-AB$ and 
$I-BA$ are $\Pc$-compact.
\end{defn}
One of the crucial and deep results (\cite[Theorem 1.16]{SeSi3}) of that theory 
states that such $\Pc$-regularizers $B$ automatically belong to $\pb{\Xb}$, hence 
$\Pc$-Fredholmness of $A$ 
is equivalent to invertibility of the respective coset $A+\pc{\Xb}$ in the quotient 
algebra $\pb{\Xb}/\pc{\Xb}$. We will address the role and importance of this 
alternative notion of Fredholmness in Section \ref{SOpSpec}. The following 
results reveal that we can even capture the (classical) Fredholm and semi-Fredholm 
properties of operators $A\in\pb{\Xb}$ in terms of $\Pc$-compact projections:
\begin{prop}(\cite[Proposition 1.27]{SeSi3} or \cite[Theorem 11]{SeFre})
Let $A\in\pb{\Xb}$. Then 
\begin{itemize}
\item $A$ is Fredholm iff there are projections $P,P'\in\pc{\Xb}$ of finite rank with
			$\im P=\ker A$ and $\ker P'=\im A$.
\item $A$ is not Fredholm iff for every $\gamma>0$ and every $k\in\Nb$ there is 
			a projection $Q\in\pc{\Xb}$ with $\rk Q\geq k$ such that 
			$\min\{\|AQ\|,\|QA\|\}\leq\gamma$. ($A$ is called $\Pc$-deficient in this case.)
\end{itemize}
\end{prop}

\begin{cor}\label{CPDeficient}
Let $A\in\pb{\Xb}$.
If $A\in\Phi_-\setminus \Phi$ (or $A\in\Phi_+\setminus \Phi$) then, for every 
$\gamma>0$ and every $k\in\Nb$, there is a projection $R\in\pc{\Xb}$ with $\rk R\geq k$ 
such that $\|AR\|\leq\gamma$ ($\|RA\|\leq\gamma$, respectively).
\end{cor}
\begin{proof}
If $A\in\Phi_-\setminus \Phi$ and $k>\dim\coker A$ then $s_k^l(A)>0$ by the 2nd 
part of Corollary \ref{CANNormSolv}. Hence
\[C:=\inf\{\|RA\|: R\text{ is a projection with }\rk R \geq k\}\geq s_k^l(A)>0.\]
By the previous proposition for every $0<\gamma<C$ there must be a $\Pc$-compact 
projection $R$ with $\rk R\geq k$ such that $\|AR\|\leq\gamma$, since 
$\|RA\|\geq C$. The case $A\in\Phi_+\setminus \Phi$ is analogous.
\end{proof}

It has already been mentioned and exploited how the (semi-)Fredholm properties translate
under passing to dual spaces and adjoint operators. Concerning the $\Pc$-notions on
$\Xb=l^p(\Zb^N,X)$, $p<\infty$, we obviously have $\Pc^*=\Pc$ on $\Xb^*=l^q(\Zb^N,X^*)$ 
where as usual $1/p+1/q=1$ if $p>1$, $q=\infty$ if $p=1$, and $q=1$ if $p=0$. From
the definitions we immediately deduce that if $A$ is $\Pc$-compact, in $\pb{\Xb}$,
or $\Pc$-Fredholm, then $A^*$ is of the same type, respectively. Thus the above
results are still useful for $A^*$.
Finally, we point out that the case $p=\infty$ is somewhat more exotic 
than the cases $p<\infty$, since there passing to the dual space and to adjoint 
operators exceeds the scale of $l^p$-spaces. 
However, for operators in $\pb{l^\infty}$ things are still under control by the 
following

\begin{prop}\label{PPFred}(cf. \cite[Corollary 1.19]{SeSi3})
Let $A\in\pb{l^\infty}$. Then $A$ maps the subspace $l^0$ to $l^0$, hence 
$A|_{l^0}\in\pb{l^0}$. Moreover, $A$ is Fredholm if and only if $A|_{l^0}$ is 
Fredholm. In this case
\[\dim\ker A = \dim\ker A|_{l^0},\quad\text{and}\quad 
\dim\coker A = \dim\coker A|_{l^0}.\]
\end{prop}

\section{On semi-Fredholm band-dominated operators}\label{SSemiF}
This section is devoted to our main question whether or when the semi-Fredholm 
property of a band-dominated operator already implies its Fredholmness.
We start with counterexamples:
\begin{ex}\label{E1}
The operator $A:l^2(\Zb^2,\Cb)\to l^2(\Zb^2,\Cb)$, 
$A:=\chi_{\Zb\times\Zb_-}I + \chi_{\Zb\times\Nb}V_{(0,1)}$
is semi-Fredholm, even one-sided invertible, but not Fredholm.
\end{ex}
\begin{ex}\label{E2}
Let $X$ be a Banach space and $\dim X=\infty$. Then 
$A:l^2(\Zb,X)\to l^2(\Zb,X)$, $A:=\chi_{\Zb_-}I + \chi_{\Nb}V_1$
is semi-Fredholm, one-sided invertible, but not Fredholm. The same observation
holds for the operator $B:=\diag\{\ldots,\Psi,\Psi,\Psi,\ldots\}$ with
$\Psi\in\lb{X}$ being a non-invertible isometry. 
\end{ex}
Thus, the claim is not true on $l^p(\Zb^N,X)$-spaces with $N>1$ or
$\dim X=\infty$. Actually these examples even demonstrate that one gets the same
outcome for the modified question with ``Fredholm'' being replaced by
``$\Pc$-Fredholm''. So, we are left with the following.

\begin{thm}\label{TSemi}
Let $\Xb:=l^p(\Zb,X)$, $p\in\{0\}\cup[1,\infty]$, $\dim X<\infty$. If $A$ is 
a semi-Fredholm band-dominated operator on $\Xb$ then $A$ is Fredholm.
\end{thm}
\begin{proof}
As a start, we consider $A\in\Phi_-\setminus\Phi$ and set $m:=\dim\coker A+1$. Then 
$\epsilon:= s_{m}^l(A)>0$ by Corollary \ref{CANNormSolv}. 
Since $A$ is band-dominated there exists an $l\in\Nb$ such that 
$\|P_{n-l} A Q_n\|\leq \epsilon/5$ for all $n>l$ (for band operators choose $l$ larger 
than the band width; and for band-dominated operators this is possible by a simple 
approximation argument). Set $d:=\dim X$.

By Corollary \ref{CPDeficient} with
the particular choice $k=2ld+m$ and $\gamma=\epsilon/(5(2^k-1))$ there is a 
$\Pc$-compact projection $R$ of rank $\tilde{k}\geq k$ such that $\|AR\|\leq\gamma$. Due 
to the $\Pc$-compactness of $R$ we find that for a sufficiently large $n\geq l+\tilde{k}$ 
also $P_nR$ has rank $\tilde{k}$ and $\|Q_nR\|\leq\gamma/\|A\|$. Thus $Y:=\im P_nR$ is a 
$\tilde{k}$-dimensional subspace of $\im P_n$ such that $\|A|_Y\|\leq 3\gamma$. 
To see the latter just notice that $\gamma\leq \|A\|/5$ and
\[\frac{\|AP_nRx\|}{\|P_nRx\|}\leq\frac{\|AP_nR\|\|Rx\|}{\|Rx\|-\|Q_nR\|\|Rx\|}\leq
\frac{\|AR\|+\|A\|\|Q_nR\|}{1-\|Q_nR\|}\leq
\frac{2\gamma}{1-1/5}\leq 3\gamma.\]
Now we consider the compressed operator $B:=P_{n-l}AP_n:\im P_n\to \im P_{n-l}$.
Then, clearly, $\|B|_Y\|\leq 3\gamma$. Thus, we get for its $k$th Bernstein number that
$B_k(B)\leq 3\gamma$, since every subspace $V$ of $\im P_n$ with 
$\dim ((\im P_n)/V) < k$ has a non-trivial intersection with $Y$. From Proposition 
\ref{PANBound} we conclude $s_k^r(B)\leq 3\gamma(2^k-1)$. Since 
$s_{\dim\im P_{n-l} - \dim\im P_n + k}^l(B)=s_k^r(B)$ 
by the definition and by Fredholm's alternative, and since
\[\dim\im P_{n-l} - \dim\im P_n + k = d(2(n-l)+1) -  d(2n+1) +  (2ld+m) = m,\]
we arrive at $s_{m}^l(B)\leq 3\gamma(2^k-1)=3\epsilon/5$.
This yields a contradiction:
\begin{align*}
s_{m}^l(A)
&=\inf\{\|A-G\|_{\lb{\Xb,\Xb}}:G\in\lb{\Xb,\Xb}, \dim\coker G\geq m\}\\
&\leq \inf\{\|A-Q_{n-l}A-P_{n-l}FP_n\|_{\lb{\Xb,\Xb}}:F\in\lb{\im P_n,\im P_{n-l}}, 
	\dim\coker F\geq m\}\\
&= \inf\{\|P_{n-l}A-P_{n-l}FP_n\|_{\lb{\Xb,\im P_{n-l}}}:F\in\lb{\im P_n,\im P_{n-l}}, 
	\dim\coker F\geq m\}\\
&\leq \inf\{\|P_{n-l}AP_n-P_{n-l}FP_n\|_{\lb{\Xb,\im P_{n-l}}}:F\in\lb{\im P_n,\im P_{n-l}}, 
	\dim\coker F\geq m\}\\
&\quad + \|P_{n-l}AQ_n\|\\
&\leq \inf\{\|B-F\|_{\lb{\im P_n,\im P_{n-l}}}:F\in\lb{\im P_n,\im P_{n-l}}, 
	\dim\coker F\geq m\} + \epsilon/5\\
&=s_{m}^l(B)+\epsilon/5<\epsilon=s_{m}^l(A).
\end{align*}
Thus, $\Phi\setminus\Phi_-$ does not contain any band-dominated operator.

Let $p\in\{0\}\cup [1,\infty)$ and $A\in\Phi_+\setminus\Phi$ be band-dominated. Then 
$A^*\in\BDO(\Xb^*)$ is in $\Phi_-\setminus\Phi$, contradicting the above, hence 
$\Phi_+\setminus\Phi$ does not contain any band-dominated operator in the cases
$p\neq\infty$. Finally, let $p=\infty$ and $A$ be $\Phi_+$. Then by Proposition 
\ref{PPFred} the restriction $A|_{l^0}$ maps into $l^0$, it is still band-dominated, 
it clearly has finite dimensional kernel and is normally solvable, hence it is 
$\Phi_+$ on $l^0$ as well. Thus, we again get that $A|_{l^0}$ is $\Phi$ by the above, 
Proposition \ref{PPFred} gives $A\in\Phi$, and the proof is finished.
\end{proof}

\begin{rem}
Note that in this proof we actually did not use that $A$ or $A^*$ are band-dominated, 
but only the property
\[\text{for every $\epsilon>0$ there is an $l\in\Nb$ such that
$\|P_{n-l} A Q_n\|, \|Q_n A P_{n-l}\|\leq \epsilon/5$ for every $n>l$}.\]
Thus, we could also consider the (more general) set of operators 
which are only subject to this condition, and we still have Theorem \ref{TSemi}. 
Such operators are called quasi-banded operators, were introduced in \cite{SeDis}
and studied in \cite[Section 4]{SeFre}, \cite{MaSaSe}. They form a Banach
subalgebra of $\pb{l^p(\Zb,X)}$ include all band-dominated operators, but also
contain e.g. Laurent and Toeplitz operators with quasi-continuous generating 
functions, as well as the flip operator $J:(x_i)\mapsto (x_{-i})$.
\end{rem}

\section{Band-dominated operators and the operator spectrum}\label{SOpSpec}
Here we will apply the above Theorem \ref{TSemi} on automatic Fredholmness in order
to discuss a remarkable improvement of one of the most important results on 
band-dominated operators, the characterization of their Fredholm property in terms 
of limit operators. We commence with the definition of $\Pc$-strong convergence and 
the operator spectrum, and we do that again in the broader framework $\pb{\Xb}$.

\subparagraph{The operator spectrum}
\begin{defn}
We say that a bounded sequence $(B_n)\subset\pb{\Xb}$ converges $\Pc$-strongly to an
operator $B\in\pb{\Xb}$ if, for every $m$, $\|P_m(B_n-B)\|+\|(B_n-B)P_m\|\to0$ as 
$n\to\infty$ and we shortly write $B=\plimn B_n$ in that case.

For an operator $A\in\pb{\Xb}$ and a sequence $h=(h_n)\subset \Zb^N$ of points which 
tend to infinity \footnote{I.e. $|h_n|\to\infty$ as $n\to\infty$}
$A_h:=\plimn V_{-h_n}AV_{h_n},$
if it exists, is called the limit operator of $A$ w.r.t. the sequence $h$.
The set
\[\sigma_{\op}(A):=\{A_h=\plimn V_{-h_n}AV_{h_n} \text{ where } 
h=(h_n)\subset\Zb^N\text{ tends to infinity}\}\]
of all limit operators of $A$ is called its operator spectrum.
We further say that $A\in\pb{\Xb}$ has a rich operator spectrum (or we simply call $A$ a rich 
operator) if every sequence $h\subset\Zb^N$ tending to infinity has a subsequence 
$g\subset h$ such that the limit operator  $A_g$ of $A$ w.r.t. $g$ exists. 
\end{defn}

Note that limit operators of limit operators of $A$ are again limit operators of $A$ 
(cf. \cite[Corollary 1.2.4]{LimOps} or \cite[Corollary 3.97]{Marko}).
The following results summarize that the operator spectrum is also stable under
passing to regularizers, adjoints and restrictions in a sense:

\begin{thm}\label{TLimOps}
(\cite[Theorem 16]{SeFre})
Let $A\in\pb{\Xb}$ be $\Pc$-Fredholm. 
Then all limit operators of $A$ are invertible and their inverses are uniformly 
bounded. Moreover, the operator spectrum of every $\Pc$-regularizer $B$ equals
\begin{equation*}
\sigma_{\op}(B) =(\sigma_{\op}(A))^{-1} := \{A_g^{-1}:A_g\in\sigma_{\op}(A)\}.
\end{equation*}
\end{thm}

\begin{prop}\label{POpSFormulas}
Let $A\in\pb{\Xb}$ be rich. If $p<\infty$ then $\sigma_{\op}(A^*)=(\sigma_{\op}(A))^*$.
If $p=\infty$ then $\sigma_{\op}(A|_{l^0})=(\sigma_{\op}(A))|_{l^0}$.
\end{prop}
\begin{proof}
Let $p<\infty$ and $A\in\pb{\Xb}$ be rich. Then, by \eqref{ELXP}, $A^*\in\pb{\Xb^*}$ 
and, by the definition of $\Pc$-strong convergence, $A_g\in\sigma_{\op}(A)$ always 
yields $(A_g)^*=(A^*)_g\in\sigma_{\op}(A^*)$. 
Conversely, if $(A^*)_h\in\sigma_{\op}(A^*)$ then, due to the richness of $A$, there 
is a subsequence $g$ of $h$ such that $A_g$ exists, and then necessarily 
$(A_g)^*=(A^*)_g=(A^*)_h$. 
Similar arguments apply to $A|_{l^0}$ in the case $p=\infty$.
\end{proof}

\subparagraph{Band-dominated operators}
It is well known and easily checked that all limit operators of band-dominated 
operators are band-dominated. The already announced main advantage of limit operators 
in the Fredholm theory for band-dominated operators reads as follows:
\begin{thm}\label{TMainOpS}
Let $A$ be a rich band-dominated operator. Then the following are equivalent
\begin{itemize}
\item[(a)] $A$ is $\Pc$-Fredholm.
\item[(b)] all limit operators of $A$ are invertible.
\item[(c)] all limit operators of $A$ are invertible and 
			$\sup\{\|A_g^{-1}\|:A_g\in\sigma_{\op}(A)\}<\infty.$
\end{itemize}
\end{thm}
\pagebreak[3]
Obviously, $\textit{(a)}\Rightarrow (\textit{(b)} \wedge \textit{(c)})$
is due to Theorem \ref{TLimOps}. The implication $\textit{(c)}\Rightarrow \textit{(a)}$
was proved in e.g. \cite{LangeR, LimOps, Marko, LiChW, HabilMarko}. 
The question whether $\textit{(b)}$ is also sufficient for the two other conditions
was an open problem, the so-called ``Big Question'', since the beginning of the story 
and was solved only recently in \cite{LiSe}. For its proof one considers the lower 
norms of the limit operators $A_g$ instead of the norms of their inverses $A_g^{-1}$.
Then the main theorem of \cite{LiSe} states that there exists a limit operator 
$C\in\sigma_{\op}(A)$ with lower norm $j(C)=\inf\{j(A_g):A_g\in\sigma_{\op}(A)\}$.
Thus, the invertibility of $C$ already implies that \textit{all} $A_g$ are bounded 
below ,i.e. $j(A_g)>0$, even \textit{uniformly}.
A remarkable ingredient in its proof is given by the following localization property for
the lower norm of band-dominated operators:

\begin{prop} \label{PLocal} (\cite[Corollary 7]{LiSe})
Let $A\in\BDO(\Xb)$ and $\delta>0$. Then there exists a $D\in\Nb$ such that 
for every operator $B\in\sigma_{\op}(A)\cup\{A\}$ and every subset $F\subset\Zb^N$ 
there exists an $x\in\Xb$ with $\|x\|=1$, $\supp x \subset F$, $\diam\supp x\leq D$, 
i.e. the support of $x$ belongs to $F$ and its diameter is not greater than $D$, 
and such that $j(B)\leq\|Bx\|\leq j(B)+\delta$.
\end{prop}
%

Here we are particularly interested in the case $\dim X<\infty$ for several reasons: 
Then it holds that band-dominated operators on $l^p(\Zb^N,X)$ with $\dim X<\infty$ are 
always rich (cf. \cite[Corollary 2.1.17]{LimOps} or \cite[Corollary 3.24]{Marko}). 
Moreover, an operator $A\in\pb{\Xb}$ is $\Pc$-Fredholm iff it is Fredholm 
in this case (The implication $\Leftarrow$ actually even holds in general 
by \cite[Corollary 12]{SeFre}, whereas the implication $\Rightarrow$ follows 
from $\pc{\Xb}\subset\lc{\Xb}$ which is a consequence of the condition $\dim X<\infty$).
Thus, Theorem \ref{TMainOpS} specifies as follows
\begin{thm}\label{TMainOpS2}
Let $A$ be a band-dominated operator on $l^p(\Zb^N,X)$ with $N\in\Nb$, 
$p\in \{0\}\cup [1,\infty]$, 
$\dim X<\infty$. Then $A$ is Fredholm if and only if all limit operators of $A$ are 
invertible. In that case their inverses are uniformly bounded.
\end{thm}

Actually, our observations of the previous section provide a further surprising
simplification and extension of this criterium:

\begin{thm}\label{TOneSided}
For a band-dominated operator $A$ on $l^p(\Zb,X)$, $p\in \{0\}\cup [1,\infty]$, 
$\dim X<\infty$, the following are equivalent
\begin{enumerate}
\item[(i)] $A$ is Fredholm.
\item[(ii)] All limit operators are invertible.
\item[(iii)] All limit operators are invertible from the left.
\item[(iv)] All limit operators are invertible from the right.
\item[(v)] All limit operators are bounded below.
\item[(vi)] All limit operators are surjective.
\end{enumerate}
In this case 
$\sup\{\|A_g^{-1}\|:A_g\in\sigma_{\op}(A)\}=(\inf\{j(A_g):A_g\in\sigma_{\op}(A)\})^{-1}<\infty.$
\end{thm}

\begin{proof}
Theorem \ref{TLimOps} provides $\textit{(i)}\Rightarrow \textit{(ii)}$ as well as the 
uniform boundedness of the inverses. Also recall that $\|B^{-1}\|=(j(B))^{-1}$ for 
every invertible operator $B$. The implication
$\textit{(ii)}\Rightarrow (\textit{(iii)} \wedge \textit{(iv)})$ is obvious. Moreover, 
$\textit{(iii)}\Rightarrow \textit{(v)}$ and $\textit{(iv)}\Rightarrow \textit{(vi)}$ 
are clear as well. 

Let all limit operators be bounded below. i.e. $\textit{(v)}$. As already mentioned, 
the main theorem of \cite{LiSe} states that there exists a $C\in\sigma_{\op}(A)$ 
with $\epsilon:=j(C)=\min\{j(A_g):A_g\in\sigma_{\op}(A)\}$. Assume that $A$ has 
infinite dimensional kernel or is not normally solvable. Taking Proposition 
\ref{PANBound} and Corollary \ref{CANNormSolv} into account, the same holds true 
for the restrictions $AQ_m:\im Q_m\to\Xb$, hence $j(AQ_m)=0$ for every $m$. 
By Proposition \ref{PLocal} there is a $D$ such that for every $m$ there exists 
$x_m\in\im Q_m$, $\|x_m\|=1$, $\diam\supp x_m \leq D$ and $\|Ax_m\|\leq \epsilon/2$. 
Now, for every $m$, choose ``centralizing shifts'' $V_{s_m}$ such that 
$Q_DV_{-s_m}x_m=0$, i.e. the shifted copies $y_m:=V_{-s_m}x_m$ belong to the 
finite dimensional space $\im P_D$. We can pass to a subsequence $(y_{m_l})$ 
which converges to a point $y\in\im P_D$. Clearly $\|y\|=1$.
Since $A$ is rich, we can pass to a subsequence of $g:=(s_{m_l})$ such that 
the respective limit operator of $A$ exists. W.l.o.g. let $g$ already be this 
sequence. But then, for sufficiently large $l$, we obtain a contradiction:
\begin{align*}
j(A_g)\leq\|A_gy\|&\leq\|(A_g-V_{-s_{m_l}}AV_{s_{m_l}})P_D\|\|y\|
+\|V_{-s_{m_l}}AV_{s_{m_l}}(y-y_{m_l})\|+\|Ax_{m_l}\|\\
&< \|(A_g-V_{-s_{m_l}}AV_{s_{m_l}})P_D\|+\|A\|\|y-y_{m_l}\|+\epsilon/2
\to_{l\to\infty}\epsilon/2< j(A_g).
\end{align*}
Thus $\textit{(v)}\Rightarrow A\in\Phi_+$ and, by Theorem \ref{TSemi},
$A\in\Phi_+\Rightarrow \textit{(i)}$ also holds and finishes this step.

Finally, let $\textit{(vi)}$ hold. Then all  $A_g$ are in $\Phi_-$ and, by 
Theorem \ref{TSemi}, all $A_g$ are Fredholm. 
In case $p<\infty$ we get that all $A_g^*$ are Fredholm and injective,
and constitute the operator spectrum of $A^*$ (cf. Proposition \ref{POpSFormulas}).
The implication $\textit{(v)}\Rightarrow \textit{(i)}$, applied 
to $A^*$, yields that $A^*$ is Fredholm, hence $A$ is Fredholm as well. In 
case $p=\infty$ we obtain from Proposition \ref{PPFred} that the restrictions 
$A_g|_{l^0}$ are Fredholm and surjective on $l^0$ as well, and these $A_g|_{l^0}$ 
constitute the operator spectrum of $A|_{l^0}$ (cf. Proposition \ref{POpSFormulas}). 
By the already proved case the restriction $A|_{l^0}$ is Fredholm 
and again by Proposition \ref{PPFred} also $A$ is Fredholm, that is we have finished 
$\textit{(vi)}\Rightarrow \textit{(i)}$.
\end{proof}

Notice that the above Examples \ref{E1}, \ref{E2} demonstrate that these one-sided 
characterizations are limited to the cases $l^p(\Zb,X)$, $\dim X<\infty$.
Further, the surjectivity in $\textit{(vi)}$ can be replaced by injectivity only in
a particular situation, but not in general:
\begin{prop}
For every band-dominated operator $A$ on $l^\infty(\Zb,X)$ with $\dim X<\infty$, the
following are equivalent
\begin{enumerate}
\item[(i)] $A$ is Fredholm.
\item[(vii)] All limit operators are injective.
\end{enumerate}
In the cases $p<\infty$ this is not true.
\end{prop}
\begin{proof}
The implications $\textit{(i)}\Rightarrow \textit{(v)}\Rightarrow \textit{(vii)}$ 
are clear for all $p$. 
The operator $A=I-V_1$ is shift invariant, hence $\sigma_{\op}(A)=\{A\}$. Furthermore,
in every case $p<\infty$, $A$ is injective on $l^p(\Zb,\Cb)$, but not Fredholm. Thus, 
$\textit{(vii)}\Rightarrow\textit{(i)}$ is false in these cases.

In the case $p=\infty$ the situation is different, for two reasons: the particular
properties of the $l^\infty$-norm, and the fact that if for a bounded sequence 
$(y_n)_n\subset l^\infty$ the truncated sequences $(P_jy_n)_n$ converge for every $j\in\Nb$
then the limits $P_jy=\lim_{n\to\infty}P_jy_n$ always define a unique element $y\in l^\infty$ 
(one says that $l^\infty$ is sequentially complete w.r.t. $\Pc$-convergence).

Assume that all limit operators of $A$ are injective, but there exists one 
$B\in\sigma_{\op}(A)$ which is not normally solvable. As in the proof of Theorem 
\ref{TOneSided} we have that for every $m$ also the restrictions $BQ_m$ are not normally 
solvable, thus there always exists an element $x_m\in\im Q_m$ with $\|x_m\|=1$, 
$\|Bx_m\|\leq 1/m$, and with bounded support. Inductively we choose sequences $g:=(m_l)$ 
and $(x_{m_l})$ of that type and such that the supports of these $x_{m_l}$ are pairwise 
disjoint. Due to the definition of the $l^\infty$-norm there exist shifted copies 
$y_{m_l}:=V_{-s_{m_l}}x_{m_l}$ such that $\|P_1y_{m_l}\|=\|x_{m_l}\|=1$.
By a Bolzano Weierstrass argument, we can pass to a subsequence $g^1=(g_l^1)$ of $g$ 
for which the sequence $(P_1y_{g^1_l})_l$ converges. Further, there is a subsequence 
$g^2$ of $g^1$ such that $(P_2y_{g^2_l})_l$ converges, etc. Continuing this, we obtain 
a family of nested sequences $g^1\supset g^2\supset\ldots\supset g^k\supset\ldots$
such that $(P_jy_{g^k_l})_l$ converge for all $1\leq j\leq k$. Define a new sequence
$h=(h_n)$ by $h_n:=g^n_n$. Then $(P_jy_{h_n})_n$ converge for each $j$ and as mentioned
above this defines a unique element $y\in l^\infty$ with 
$P_jy=\lim_{n\to\infty}P_jy_{h_n}$ for every $j$ and, by the construction of the $y_{h_n}$, 
also $\|y\|=1$.

Moreover, there is a subsequence of $(s_{h_n})_n$ such that the limit operator $C$ 
of $B$ w.r.t. this subsequence exists. W.l.o.g. let already $(s_{h_n})$ be this 
subsequence. Then $C$ is also a limit operator of $A$, and our final step is to prove 
$Cy=0$, which yields a contradiction. Indeed, with $t_n:=s_{h_n}$, for all $j\in\Nb$
\begin{align*}
\|P_jCy\|
&=\|P_j(C-V_{-t_n}BV_{t_n}+V_{-t_n}BV_{t_n}Q_{k})y
+P_jV_{-t_n}BV_{t_n}P_{k}(y-y_{h_n})+P_jV_{-t_n}BV_{t_n}P_{k}y_{h_n}\|\\
&\leq\|P_j(C-V_{-t_n}BV_{t_n})\|+\|P_jV_{-t_n}BV_{t_n}Q_{k}\|
+\|B\|\|P_{k}(y-y_{h_n})\|+\|P_jV_{-t_n}BV_{t_n}P_{k}y_{h_n}\|.
\end{align*}
Clearly the first term tends to zero as $n\to\infty$. The second one is smaller than 
any prescribed positive $\epsilon$ if $k$ is sufficiently large (note that $B$ is 
band-dominated hence can be approximated in the norm by band operators), and the 
third term goes to zero as well for every fixed $k$ as $n\to\infty$. The last 
one can be estimated by
$\|P_jV_{-t_n}BV_{t_n}P_{k}y_{h_n}\|\leq\|P_jV_{-t_n}BV_{t_n}Q_{k}\|+\|Bx_{h_n}\|$. 
Here the first part is again smaller than any prescribed positive $\epsilon$ if $k$ 
is sufficiently large, and the last one tends to $0$ as $n\to\infty$. This finishes
the proof.
\end{proof}
This characterization $\textit{(vii)}$ in the $l^\infty$-case is known as Favard's 
condition, and was studied even in more general situations, e.g. for $A=I-K$ on 
$l^\infty(\Zb,X)$ with arbitrary Banach spaces $X$ and $K$ having collectively compact 
entries, see \cite{LiChWFavard, LiChW, SeFre}. Probably, also the observations of the 
present work have respective extensions for such operators $A=I-K$, but this shall be 
a topic of some future work.

\begin{ex}
Finally, one might ask, whether
\textit{
\begin{itemize}
\item[(viii)] All limit operators are one sided invertible.
\end{itemize}}
\noindent
i.e. a mixed version of $\textit{(iii)}$ and $\textit{(iv)}$, may also be an equivalent 
condition. 
This is false as the following example demonstrates: Let $I_n,U_n\in\Cb^{n\times n}$ 
be the identity matrix and the shift by one position 
to the right, respectively. Define $A:=\diag(\chi_{-}I,I_1,U_1,I_2,U_2,I_3,U_3,\ldots)$ on
the space $l^2(\Zb,\Cb)$. Then the operator spectrum of $A$ consists of the identity $I$, 
the bilateral shift $V_1$ and all shifted copies of the two operators $\chi_-V_1 +\chi_+I$, 
$V_1\chi_+I +\chi_-I$. Clearly, each of these limit operators is one sided invertible, but 
$A$ is not Fredholm.
\end{ex}



\end{document}